\documentclass[11pt,centertags,leqno]{amsart}
\usepackage[latin1]{inputenc}
\usepackage[english]{babel}
\usepackage{cite}
\usepackage{verbatim}
\usepackage{color}
\usepackage{pgf,tikz}
\usetikzlibrary{arrows,shapes,snakes}
\usepackage{yhmath}
\usepackage{graphicx,picture}
\usepackage{amsmath,amsthm,amsopn,amstext,amscd,amsfonts,amssymb}
\usepackage{comment}
\usepackage{hyperref}
\usepackage{enumitem}
\setlist[enumerate]{leftmargin=*}
\usepackage[top=2.5cm, bottom=2.5cm, left=3cm,right=3cm]{geometry}

\newcommand{\diesis}{^\#}

\newtheorem{theo}{Theorem}[section]
\newtheorem{lemma}{Lemma}[section]
\newtheorem{prop}{Proposition}[section]
\newtheorem{cor}{Corollary}[section]

\theoremstyle{definition}

\newtheorem{rem}{Remark}[section]

\numberwithin{equation}{section}% MATH -----------------------------------------------------------

\newcommand{\R}{\mathbb R}
\newcommand{\de}{\partial}
\newcommand{\eps}{\varepsilon}

\newcommand{\ds}{\displaystyle}
\newcommand{\mc}{\mathcal}
\newcommand{\dx}{\,dx}

\DeclareMathOperator{\esssup}{ess\,sup}

\begin{document}
\title[]{Some remarks on a shape optimization problem}
\author[F. Della Pietra]{Francesco Della Pietra}
\address{Francesco Della Pietra \\
 U\-ni\-ver\-si\-t\`a de\-gli stu\-di di Na\-po\-li ``Fe\-de\-ri\-co II''\\
Di\-par\-ti\-men\-to di Ma\-te\-ma\-ti\-ca e Ap\-plica\-zio\-ni 
``R.~Cac\-ciop\-po\-li''\\
Via Cintia, Complesso di Monte Sant'An\-ge\-lo\\
I-80126 Napoli, Italia.
}
\email{f.dellapietra@unina.it}

\keywords{Eigenvalues; Shape optimization; Symmetrization}
\subjclass[2010]{35P15, 49R50}
\date{{\today}}
\maketitle

\begin{abstract} Given $\Omega$ bounded open set of $\R^{n}$ and $\alpha\in \R$, let us consider
\[
\mu(\Omega,\alpha)=\min_{\substack{v\in W_{0}^{1,2}(\Omega)\\v\not\equiv 0}} 
\frac{\ds\int_{\Omega} |\nabla v|^{2}dx+\alpha \left|\ds\int_{\Omega}|v|v\,dx \right|}{\ds\int_{\Omega} 
|v|^{2}dx}.
\] 
We study some properties of $\mu(\Omega,\alpha)$ and of its minimizers, and, depending on $\alpha$, we determine the set $\Omega_{\alpha}$ among those of fixed measure such that $\mu(\Omega_{\alpha},\alpha)$ is the smallest possible.
\end{abstract}

\section{Statement of the problem and main result}
Let $\Omega$ be a bounded open set of $\R^{n}$, $n\ge 2$,
and consider the following minimum problem
\begin{equation}
	\label{pb}
\mu(\Omega,\alpha)=\min_{\substack{v\in W_{0}^{1,2}(\Omega)\\v\not\equiv 0}} \mc Q(v,\alpha)
\end{equation}
where  $\alpha$ is a fixed real number and
\[
\mc Q(v,\alpha)=\frac{\ds\int_{\Omega} |\nabla v|^{2}dx+\alpha \left|\ds\int_{\Omega}|v|v\,dx \right|}{\ds\int_{\Omega} 
|v|^{2}dx}.
\] 
The objective of this paper is to study some properties of $\mu(\Omega,\alpha)$ and of its minimizers. Moreover, we aim to determine and to characterize the sets $\tilde\Omega$ among those of fixed measure such that $\mu(\tilde\Omega,\alpha)$ is the smallest possible. As we will show, the shape of $\tilde\Omega$ depends on 
$\alpha$. More precisely if we denote, as usual, by $\omega_{n}$ the measure of the unit ball in $\R^{n}$, and by $j_{n/2-1,1}$ the first zero of the Bessel function of first kind of order $n/2-1$, the main result of the paper is the following.
\begin{theo}
\label{mainthm}
Let $n\ge 2$. There exists a positive number
\[
\alpha_{c} = j^{2}_{n/2-1,1}\omega_{n}^{\frac 2n}
\left[2^{\frac 2 n} -1\right]
\]
such that, for every bounded open set $\Omega\subset \R^{n}$ and for every $\alpha\in \R$, it holds 
\begin{equation}
\label{mainin}
	 \mu (\Omega,\alpha)
		\ge
	\begin{cases}
	\mu(\Omega\diesis,\alpha) &\text{if }\alpha|\Omega|^{\frac 2n}\le \alpha_{c},\\[.2cm]
	\frac{2^{2/n}\omega_{n}^{2/n} j_{n/2-1,1}^{2}}{|\Omega|^{2/n}}& \text{if }	\alpha|\Omega|^{\frac 2n}\ge \alpha_{c},
	\end{cases}
\end{equation}
where $\Omega\diesis$ is the ball centered at the origin with measure $|\Omega\diesis|=|\Omega|$. If the equality sign holds when $\alpha|\Omega|^{\frac 2n}< \alpha_{c}$, then $\Omega$ is a ball. If the equality sign holds when $\alpha|\Omega|^{\frac 2n}> \alpha_{c}$, then $\Omega$ is  the union of two disjoint balls of equal measure.
If $\alpha|\Omega|^{\frac 2n}=\alpha_{c}$ and the equality sign holds, $\Omega$ is a ball or the union of two disjoint balls of equal measure.
\end{theo}
On the other hand, the above result provides the best constant $\mu(\tilde\Omega,\alpha)$ in the corresponding Sobolev-Poincar\'e inequality:
\[
\mu(\tilde\Omega,\alpha)\int_{\Omega}|v|^{2} \dx \le \int_{\Omega} |\nabla v|^{2}\dx +\alpha \left|\int_{\Omega} |v|v\dx\right|,\quad v\in W_{0}^{1,2}(\Omega),
\]
among all the open bounded sets $\Omega$ with fixed measure.

Let us observe that when $\alpha=0$, the above inequality reduces to the classical Poincar\'e inequality. Moreover, $\mc Q(\Omega,0)$ is the Rayleigh quotient associated to the Dirichlet Laplacian eigenvalue problem, and $\mu(\Omega,0)$ corresponds to its first eigenvalue in $\Omega$. Then, it is well known the Faber-Krahn inequality:
\[
\mu(\Omega,0)\ge \mu(\Omega\diesis,0).
\]
Moreover if the equality replaces the inequality, then $\Omega$ is a ball. 

The problem of finding the optimal shape of set-dependent functionals is largely studied in many settings. Several results can be found for example in \cite{hen}, related to eigenvalue problems, or in \cite{ma}. Recent results are contained for example in 
\cite{brkj,bfnt11,efknt,gn11,bdnt14,dpgrobin,%
dpgtors,dpgccm,dpgmonge,dpg2,dpg1,n12,chdb12}. Moreover, we recall that in \cite{bfnt11} a result analogous to Theorem \ref{mainthm} is given for the functional
\[
\tilde\lambda(\Omega,\alpha)=\min_{\substack{v\in W_{0}^{1,2}(\Omega)\\v\not\equiv 0}} 
\frac{\ds\int_{\Omega} |\nabla v|^{2}dx+\alpha \left(\ds\int_{\Omega}v\,dx \right)^{2}}{\ds\int_{\Omega} 
|v|^{2}dx},
\]
which is related to a nonlocal eigenvalue problem. It has been proved that there exists a threshold positive value $\tilde \alpha$ such that if $\alpha<\tilde\alpha$, the minimum of $\tilde \lambda(\Omega,\alpha)$ among the sets with fixed measure is attained at one ball, while for $\alpha$ greater than $\tilde \alpha$, such minimum is given at two balls of equal measure.

The paper is organized as follows. In Section 2, we recall some basic results on Schwarz symmetrization and on the Dirichlet Laplacian. Moreover, depending on $\alpha$, we give some properties of $\mu(\Omega,\alpha)$ and its minimizers. 
Finally, in Section 3 we give the proof of the main result.

\section{Notation and preliminary results}
\subsection{Schwarz symmetrization}
Let $\Omega$ be a bounded open set of $\R^{n}$ and $u\colon\Omega\rightarrow \R$ be a measurable function, and denote by $\Omega\diesis$ the ball centered at the origin with the same Lebesgue measure of $\Omega$. The Schwarz rearrangement of $u$ is the spherically symmetric decreasing function
\[
u\diesis\colon \Omega\diesis\rightarrow [0,+\infty[
\]
whose level sets are balls have the same measure of the level sets of $|u|$, that is 
\[
|\{u\diesis>t\}|=|\{ |u|>t\}|, \quad t\ge 0.
\]
The Schwarz symmetrization enjoys the following properties.
\begin{enumerate}[label=\emph{\alph*}),ref=\emph{\alph*}]
\item By definition, $u\diesis$ preserves the $L^{p}$-norm of $u$:
\begin{equation}
\label{lp}
\|u\|_{L^{p}(\Omega)}= \|u\diesis\|_{L^{p}(\Omega\diesis)},\quad 1\le p \le +\infty.
\end{equation}
\item The P\'olya-Szego inequality holds: if $u\in W_{0}^{1,2}(\Omega)$ is a nonnegative function, then
\begin{equation}
	\label{ps}
	\int_{\Omega} |\nabla u|^{2}\dx \ge 	\int_{\Omega\diesis} |\nabla u\diesis|^{2}\dx.
\end{equation}
Moreover, if the above inequality becomes an equality, and
\[
|\{|\nabla u\diesis|=0\}\cap (u\diesis)^{-1}(0,\esssup u)|=0,
\]
then, up to translations, $\Omega=\Omega\diesis$ and $u=u\diesis$ almost everywhere (see \cite{bz88}). 

For an exhaustive treatment on rearrangements and symmetrization, we refer the reader, for example, to \cite{k}. 
\end{enumerate}
\subsection{Some basic facts for the Dirichlet Laplacian}
Given $G\subset \R^{n}$ bounded open set, throughout the paper we will denote by $\lambda_\Delta (G)$ the first Dirichlet-Laplace eigenvalue relative to $G$: 
\begin{equation}
\label{dirlap}
 \lambda_\Delta (G)= \min_{v\in W_{0}^{1,2}(G)\setminus\{0\}} \frac{\ds\int_{G} |\nabla v|^{2}dx}{\ds\int_{G} |v|^{2}dx},
\end{equation}
and by $\lambda_{T}(G)$ the minimum of the constrained problem
\begin{equation}
\label{twistedef}
 \lambda_{T}(G)= \min_{\substack{v\in W_{0}^{1,2}(G)\setminus\{0\} \\ \int_{G}|v|v\dx=0}} \frac{\ds\int_{G} |\nabla v|^{2}dx}{\ds\int_{G} |v|^{2}dx}.
\end{equation}

As regards \eqref{dirlap}, we recall some basic properties:

\begin{enumerate}
\item The Faber-Krahn inequality: for any bounded open set in $\Omega\subset\R^{n}$, it holds that
\[
\lambda(\Omega)\ge \lambda(\Omega\diesis)=\frac{\omega_{n}^{2/n}}{|\Omega|^{2/n}}
j_{n/2-1,1}^{2},
\]
where $j_{n/2-1,1}$ denotes, as usual, the first zero of the Bessel function of first kind of order $n/2-1$. If equality sign holds, then $\Omega$ is a ball.
\item If $\Omega=B_{1}\cup B_{2}$ is the union of two disjoint balls $B_{1},B_{2}$ with different radii $R_{1}>R_{2}>0$, then 
\[
\lambda_{\Delta}(\Omega)=\lambda_{\Delta}(B_{1})=\frac{j_{n/2-1,1}^{2}}{R_{1}^{2}}.
\] 
Hence it is simple, any associated eigenfuction does not change sign in the largest ball $B_{1}$, and it is identically zero in $B_{2}$.
\item If $\Omega=B_{1}\cup B_{2}$ is the union of two disjoint balls $B_{1},B_{2}$ with equal radii $0<R_{1}=R_{2}$, the first eigenvalue is not simple, and there exists an eigenfunction $u$ positive in $B_{1}$, negative in $B_{2}$ and such that $\int_{B_{1}\cup B_{2}}|u|u\dx =0$. In particular, this eigenfunction coincides with the positive first eigenfunction of $\lambda_{\Delta}(B_{1})$, and to its opposite (up to a translation) in $B_{2}$.
\end{enumerate}

\subsection{Some properties of $\mu(\Omega,\alpha)$}
In what follows, for a given function $u\colon\Omega\rightarrow \R$, $u_{+}=\max\{u,0\}$ and $u_{-}=\max\{-u,0\}$ will be its positive and negative part, and 
\[
	\Omega_{+}=\{u_{+}>0\},\qquad 	\Omega_{-}=\{u_{-}>0\}.
\]
\begin{prop}
The following properties for $\mu(\Omega,\alpha)$ hold.
\begin{enumerate}[label={(\alph*)},ref=\emph{\alph*}]
\item
The minimum $\mu(\Omega,\alpha)$ is 1-Lipschitz continuous and it is non-decreasing with respect to $\alpha\in\R$. 
\item For $\alpha < 0$, 
\[
\mu(\Omega,\alpha)=\lambda_\Delta (\Omega)+\alpha.
\]
\item For $\alpha\ge 0$, 
\begin{equation}
\label{ub}
\lambda_\Delta (\Omega)\le\mu(\Omega,\alpha)\le \min\{\lambda_{T}(\Omega),\,\lambda_\Delta (\Omega)+\alpha\}.
\end{equation}
\item As $\alpha\rightarrow +\infty$, we have that
\begin{equation*} 
\lim_{\alpha\rightarrow +\infty}\mu(\Omega,\alpha)=\lambda_{T}(\Omega).
\end{equation*}
\end{enumerate}
\end{prop}
\begin{proof}
\begin{enumerate}[label={(\alph*)},ref=\emph{\alph*}]
\item For any $\eps>0$,
\[
\mc Q(v,\alpha)\le \mc Q(v,\alpha+\eps)\le \mc Q(v,\alpha)+\eps.
\]
Taking the minimum over $W_{0}^{1,2}(\Omega)\setminus \{0\}$, we have
\[
0\le \mu(\Omega,\alpha+\eps)-\mu(\Omega,\alpha)\le \eps,
\]
and the proof of (a) is concluded.
\item
Being $\alpha<0$, we have that $\mc Q(v,\alpha)\ge \mc Q(|v|,\alpha) =\mc Q(v,0)+\alpha \ge \lambda_\Delta (\Omega)+\alpha$, for any $v \in W_{0}^{1,2}(\Omega)$. On the other hand, if  $u\in W_{0}^{1,2}(\Omega)$ is a nonnegative minimizer for \eqref{dirlap}, $\mc Q(u,\alpha)=\lambda_{\Delta}(\Omega)+\alpha$, and then necessarily 
 $\lambda_{\Delta}(\Omega)+\alpha=\mu(\Omega,\alpha)$.
\item It follows immediately from the definitions of $\mu,\lambda_{\Delta}$ and $\lambda_{T}$.
\item Let $0\le \alpha_{k}$, $k\in \mathbb N$, be a positively divergent sequence. For any $k$, consider a minimizer $u_{k}\in W_{0}^{1,2}(\Omega)$ of \eqref{pb}
such that $\|u_{k}\|_{2}=1$. We have that
\[
 \mu(\Omega,\alpha_{k}) = \int_{\Omega}|\nabla u_{k}|^{2}dx +\alpha_{k}\left|\int_{\Omega}|u_{k}|u_{k}\dx\right| \le \lambda_{T}(\Omega).
\]
Then $u_{k}$ converges (up to a subsequence) to a function $U\in W_{0}^{1,2}(\Omega)$ strongly in $L^{2}(\Omega)$ and weakly in $W_{0}^{1,2}(\Omega)$. Moreover, $\|U\|_{L^{2}(\Omega)}=1$ and
\[
\left| \int_{\Omega}|u_{k}|u_{k}\right| \le \frac{\lambda_{T}(\Omega)}{\alpha_{k}}\rightarrow 0\quad\text{ as }k\rightarrow +\infty,
\]
which gives that $\int_{\Omega}|U|U\dx=0$. On the other hand, the weak convergence in $W_{0}^{1,2}(\Omega)$ implies that
\[
\int_{\Omega}|\nabla U|^{2}\dx \le \liminf \int_{\Omega}|\nabla u_{k}|^{2}\dx.
\]
Finally, by definition of $\lambda_{T}(\Omega)$, and \eqref{ub} we have
\begin{equation*}
\begin{split}
\lambda_{T}(\Omega) \le\ds \int_{\Omega}& |\nabla U|^{2}\dx \le \\ &\le 
\ds\liminf_{k\rightarrow +\infty}
 \left(\int_{\Omega}|\nabla u_{k}|^{2}\dx +\alpha_{k}\left|\int_{\Omega}|u_{k}|u_{k}\dx\right|\right) = \lim_{k\rightarrow+\infty}\mu(\Omega,\alpha_{k})\le \lambda_{T}(\Omega),
 \end{split}
\end{equation*}
and the proof is completed. 
\end{enumerate}
\end{proof}
\begin{rem}
Let us observe that from the above proposition, (b) gives that $\mu(\Omega,\,\cdot\,)$ is unbounded from below. Moreover, $\mu(\Omega,\alpha)=0$ corresponds to $-\alpha=\lambda(\Omega)$.  
\end{rem}
\begin{rem}
Among the properties of $\mu(\Omega,\alpha)$, we observe also that it does not have the same behavior of the usual Dirichlet Laplacian with respect to the rescaling of the domain, being also the term $\alpha$ affected of the rescaling. Indeed, while $\lambda_{\Delta}(t\Omega)=t^{-2}\lambda(\Omega)$, it holds that
$\mu(t\Omega;\alpha)= t^{-2}\mu(\Omega;t^{2}\alpha)$.
\end{rem}
In the proposition below, we describe some features of $\mu(\Omega,\alpha)$ by computing the associated Euler equation.
Without loss of generality we may assume that a minimizer $u$ satisfies $\int_{\Omega}|u|u\dx\ge 0$. 
\begin{lemma}
 Let $\alpha\ge 0$, and  $u\in W_{0}^{1,2}(\Omega)$ be a minimizer for \eqref{pb}. Then $u_{+}\in W_{0}^{1,2}(\Omega_{+})$ and $u_{-}\in W_{0}^{1,2}(\Omega_{-})$ are first eigenfunctions of the Dirichlet Laplacian relative to $\Omega_{+}$ and $\Omega_{-}$ respectively. Moreover:
\begin{enumerate}
 	\item suppose that $\ds \int_{\Omega}|u|u\dx>0$.
 	\begin{enumerate}
		\item If $u_{-}\equiv 0$ in $\Omega$, then
			\begin{equation}
			\label{solnoneg}
			\mu(\Omega,\alpha)=\lambda_\Delta (\Omega_{+})+\alpha.
			\end{equation}
		\item If $u_{-}\not \equiv 0$ in $\Omega$, then
			\begin{equation} 
			\label{media}
				\mu(\Omega,\alpha)=\lambda_\Delta (\Omega_{+})+\alpha=\frac{\lambda_\Delta (\Omega_{+})+\lambda_\Delta (\Omega_{-})}{2},
			\end{equation}
			and then the parameter $\alpha$ corresponds to
		\begin{equation}
			\label{media2}
			\alpha= \frac{\lambda_\Delta (\Omega_{-})-\lambda_\Delta (\Omega_{+})}{2}.
		\end{equation}
	\end{enumerate}
	In both cases (a) and (b),
	\begin{equation}
	\label{equality}
		\lambda_\Delta (\Omega_{+})=\lambda_\Delta (\Omega),
	\end{equation}
	and
	\begin{equation}
		\label{equality2}
		\mu(\Omega,\alpha)=\lambda_{\Delta}(\Omega)+\alpha=\lambda_{\Delta}(\Omega_{+})+\alpha.
	\end{equation}
	\item  Suppose that $\ds\int_{\Omega}|u|u\dx= 0$. Then
	\begin{equation} 
		\label{twisted}
		\mu(\Omega,\alpha)=\lambda_{T}(\Omega)=\frac{\lambda_\Delta (\Omega_{+})+
		\lambda_\Delta (\Omega_{-})}{2}.
	\end{equation}
	More precisely, if there exists $\bar \alpha$ such that a minimizer $\bar u$ of $\mu(\Omega,\bar\alpha)$ satisfies $\int_{\Omega}|\bar u|\bar u\dx=0$, then for any $\alpha>\bar \alpha$, $\bar u$ is a minimizer for $\mu(\Omega,\alpha)$, the equality in \eqref{twisted} holds, and $\bar u$ is a minimizer also for $\lambda_{T}(\Omega)$.
\end{enumerate}   
\end{lemma}
\begin{proof}
For sake of simplicity, here we write $\mu=\mu(\Omega,\alpha)$, and distinguish two cases. \\
\noindent {\bf Case 1:} {$\ds\int_{\Omega}|u|u\dx> 0$}. We have that $u$ solves
\begin{equation*}
	\left\{
		\begin{array}{ll}
			-\Delta u =  \mu u - \alpha |u| &\text{in } \Omega, \\
			u=0 & \text{on }\de \Omega.
		\end{array}
	\right.
\end{equation*}
If $u\ge 0$ in $\Omega$, then $u_{+}$ satisfies
\begin{equation*}
	\left\{
		\begin{array}{ll}
			-\Delta u_{+} =  (\mu - \alpha) u_{+} &\text{in } \Omega_{+}, \\
			u_{+}=0 & \text{on }\de \Omega_{+}.
		\end{array}
	\right.
\end{equation*}
The positivity of the eigenfunction $u_{+}$ in $\Omega_{+}$ guarantees that $\mu-\alpha$ coincides with the first eigenvalue $\lambda_{+}(\Omega)$ on $\Omega_{+}$, and then \eqref{solnoneg} holds. Moreover, \eqref{equality} follows from the inequalities
\[
 \lambda_\Delta (\Omega_{+})+\alpha\le \lambda_\Delta (\Omega)+\alpha\le \lambda_\Delta (\Omega_{+})+\alpha,
\] 
obtained by substituting \eqref{solnoneg} in \eqref{ub}, and recalling the monotonicity of the Dirichlet-Laplace eigenvalues with respect to the inclusion of sets.

If $u$ changes sign in $\Omega$, then $u_{+}$ and $u_{-}$ satisfy
\begin{equation*}
	\left\{
		\begin{array}{ll}
			-\Delta u_{+} =  (\mu - \alpha) u_{+} &\text{in } \Omega_{+}, \\
			u=0 & \text{on }\de \Omega_{+},
		\end{array}
	\right.
\text{
and
}
	\left\{
		\begin{array}{ll}
			-\Delta u_{-} =  (\mu + \alpha) u_{-} &\text{in } \Omega_{-}, \\
			u_{-}=0 & \text{on }\de \Omega_{-}.
		\end{array}
	\right.
\end{equation*}
Hence
\[
\lambda_\Delta (\Omega_{+})=\mu-\alpha,\quad \lambda_\Delta (\Omega_{-})=\mu+\alpha,
\]
that give \eqref{media} and \eqref{media2}. Similarly as before, substituting \eqref{media} and \eqref{media2} in \eqref{ub} and using the monotonicity of $\lambda_\Delta (\,\cdot\,)$, the equality \eqref{equality} holds. By \eqref{solnoneg}, \eqref{media} and \eqref{equality} we get also \eqref{equality2}. 

\noindent {\bf Case 2:} {$\ds\int_{\Omega}|u|u\dx=0$}. First, we observe that in this case
\begin{equation*}
\mu(\Omega,\alpha)= \lambda_{T}(\Omega).
\end{equation*}
Indeed, by definition of $\mu$ and $\lambda_{T}$, and being $u$ an admissible function for \eqref{twistedef}, we have
\[
\lambda_{T}(\Omega)\ge \mu(\Omega,\alpha)=\mc Q(u,\alpha)\ge \lambda_{T}(\Omega).
\]
Computing the Euler equation with the constraint $\int_{\Omega}|u|u\dx=0$, the functions $u_{+}\in W_{0}^{1,2}(\Omega_{+})$ and $u_{-}\in W_{0}^{1,2}(\Omega_{-})$ satisfy
\begin{equation}
\label{eul+}
\left\{
\begin{array}{ll}
-\Delta u_{+} = \lambda_{+}
\,u_{+} &\text{in }\Omega_{+},\\[.1cm]
u_{+}=0&\text{on }\de \Omega_{+},
\end{array}
\right.
\text{ and }
\left\{
\begin{array}{ll}
-\Delta u_{-} = \lambda_{-}
\,
u_{-} &\text{in }\Omega_{-},\\[.1cm]
u_{-}=0&\text{on }\de \Omega_{-},
\end{array}
\right.
\end{equation}
for some positive values $\lambda_{+}$ and $\lambda_{-}$ (see also \cite{n12}). Moreover, being $u_{+}$ and $u_{-}$ positive functions in $\Omega_{+}$ and $\Omega_{-}$ respectively, it follows that
\[
\lambda_{+}=\lambda_\Delta (\Omega_{+}),\quad\lambda_{-}=\lambda_\Delta (\Omega_{-}).
\]
Hence, in this case we have that $\int_{\Omega_{+}}u_{+}^{2}\dx=\int_{\Omega_{-}}u_{-}^{2}\dx$, and from the minimality of $u$ and using \eqref{eul+} it follows that
\[
\mu(\Omega,\alpha) = \mc Q(u,\alpha)=
\frac{\ds\int_{\Omega_{+}}|\nabla u_{+}|^{2}\dx+\int_{\Omega_{-}}|\nabla u_{-}|^{2}\dx}{\ds\int_{\Omega_{+}}u_{+}^{2}\dx+\int_{\Omega_{-}}u_{-}^{2}\dx}=
 \frac{\lambda_\Delta (\Omega_{+})+\lambda_\Delta (\Omega_{-})}{2},
\]
and \eqref{twisted} is proved. The proof of (2) is completed by recalling that the function $\mu(\Omega,\,\cdot\,)$ is nondecreasing and bounded from above by $\lambda_{T}(\Omega)$.
\end{proof}
Using the above lemma, the minimum $\mu(\Omega,\alpha)$ can be characterized as follows.
\begin{prop}
\label{charomega}
If $\Omega$ is a bounded open set of $\R^{n}$, then
\[
\mu(\Omega,\alpha) = \min\{\lambda_\Delta (\Omega)+\alpha, \lambda_{T}(\Omega)\} =
\begin{cases}
\lambda_\Delta (\Omega)+\alpha,&\text{if }\alpha \le \lambda_{T}(\Omega)-\lambda_\Delta (\Omega),
\\
\lambda_{T}(\Omega), &\text{if }\alpha \ge \lambda_{T}(\Omega)-\lambda_\Delta (\Omega).
\end{cases}
\]
\end{prop}
\begin{proof}
Let $\alpha\ge 0$ be fixed. We have to show that $\mu(\Omega,\alpha)=\min\{\lambda_{\Delta}(\Omega)+\alpha,\lambda_{T}(\Omega)\}$.

Clearly if $\lambda_{\Delta}(\Omega)+\alpha< \lambda_{T}(\Omega)$, a minimizer $u$ of  $\mu(\Omega,\alpha)$ cannot verify $\int_{\Omega}|u|u\dx=0$. Otherwise, by \eqref{twisted}, and choosing a nonnegative first eigenfunction $u_{1}$ of $-\Delta$ in $\Omega$, we have
\[
Q(u,\alpha)=\mu(\Omega,\alpha)=\lambda_{T}(\Omega)>\lambda_{\Delta}(\Omega)+\alpha=\mc Q(u_{1},\alpha),
\]
contradicting the minimality of $u$. Hence $\int_{\Omega}|u|u\dx>0$, and by \eqref{equality2}, $\mu(\Omega,\alpha)=\lambda_{\Delta}(\Omega)+\alpha$.

Analogously, if $\lambda_\Delta (\Omega)+\alpha > \lambda_{T}(\Omega)$, a minimizer $u$ necessarily satisfies $\int_{\Omega}|u|u\dx=0$, and $\mu(\Omega,\alpha)=\lambda_{T}(\Omega)$. 
\end{proof}
\begin{rem}
We explicitly observe that if $\Omega$ is connected, a minimizer $u$ of 
$\mu(\Omega,\alpha)$ either is positive in $\Omega$ or $\int_{\Omega}|u|u\dx =0$.
\end{rem}
Assuming now that $\Omega$ is the union of two disjoint balls (possibly one ball),
we have the following.
\begin{cor}
\label{charballs}
If $\Omega=B_{1}$, with radius $R_{1}>0$, then
\begin{equation}
\label{mupalla}
\mu(B_{1};\alpha)=
\begin{cases}
\frac{j^{2}_{n/2-1,1}}{R_{1}^{2}}+\alpha,&\text{if }\alpha\le \lambda_{T}(B_{1})-\frac{j^{2}_{n/2-1,1}}{R_{1}^{2}},\\[.2cm]
\lambda_{T}(B_{1})&\text{otherwise}.
\end{cases}
\end{equation}
If $\Omega=B_{1}\cup B_{2}$, where $B_{1}$ and $B_{2}$ are disjoint balls with radii $R_{1}$, $R_{2}$ such that $R_{1}\ge R_{2}>0$, then
\begin{equation}
\label{muduepalle}
\mu(B_{1}\cup B_{2},\alpha)=
\begin{cases}
\frac{j^{2}_{n/2-1,1}}{R_{1}^{2}}+\alpha,&\text{if }\alpha\le \lambda_{T}(B_1\cup B_{2})-\frac{j^{2}_{n/2-1,1}}{R_{1}^{2}},\\[.2cm]
\lambda_{T}(B_1\cup B_{2})&\text{otherwise}.
\end{cases}
\end{equation}

In particular, if $R_{1}=R_{2}$, for any $\alpha\ge 0$
\begin{equation}
\label{mupaug}
\mu(B_{1}\cup B_{2};\alpha)= \lambda_{T}(B_{1}\cup B_{2})=\frac{2^{\frac 2 n} \omega_{n}^{\frac 2 n}j^{2}_{\frac n 2 -1,1} }{|\Omega|^{\frac 2 n}},
\end{equation}
where the value in the right-hand side is $\lambda_\Delta (B_{1}\cup B_{2})$, and any minimizer of $\mu(B_{1}\cup B_{2},\alpha)$ is a minimizer of $\lambda_{T}(B_{1}\cup B_{2})$.

\end{cor}
\begin{proof}
The proof of \eqref{mupalla} and \eqref{muduepalle} follows from Proposition \ref{charomega} by writing explicitly $\lambda_{\Delta}$ in the case of one ball or two disjoint balls. Then, we have only to show last equality in \eqref{mupaug}. Observe first that
\[
\lambda_{T}(B_{1}\cup B_{2})\ge \lambda_\Delta (B_{1}\cup B_{2}).
\]
On the other hand, being $B_{1}$ and $B_{2}$ disjoint balls with equal measure, there exists an eigenfunction $V$ of the Dirichlet Laplacian relative to $B_{1}\cup B_{2}$ such that $\int_{B_{1}\cup B_{2}}|V|V \dx =0$. More precisely, this eigenfunction corresponds to a first positive Dirichlet Laplacian eigenfunction on $B_{1}$, and to its opposite (up to a translation) on $B_{2}$. Then $V$ is an admissible test function for the Rayleigh quotient of $\lambda_{T}(B_{1}\cup B_{2})$, and
\[
\lambda_{T}(B_{1}\cup B_{2}) \le \frac{\ds\int_{B_{1}\cup B_{2}}|\nabla\, V|^{2}\dx}{\ds\int_{B_{1}\cup B_{2}}V^{2}\dx}= \lambda_\Delta (B_{1}\cup B_{2}),
\]
and then $\lambda_\Delta (B_{1}\cup B_{2})=\lambda_{T}(B_{1}\cup B_{2})$.
\end{proof}
\section{Proof of Theorem \ref{mainthm}}
The proof of Theorem \ref{mainthm} will be pursued in two main steps. First, we show that the minimum of $\mu(\Omega,\alpha)$ among all sets of fixed measure is reached at the union of two disjoint balls. Second, we minimize among such sets.
 
\subsection{An isoperimetric inequality for $\mu(\Omega,\alpha)$}
The first step in order to prove Theorem \ref{mainthm} is to show an 
isoperimetric inequality for $\mu(\Omega,\alpha)$. To this aim, let
\[
\mathcal B(|\Omega|) =\{A=B_{1}\cup B_{2}\colon B_{1},B_{2}\text{ open disjoint balls of }\R^{n},\; |B_{1}\cup B_{2}|=|\Omega|\}.
\]
In the above definition we are implicitly assuming that $A\in \mc B(|\Omega|)$ can be a unique ball.
\begin{prop}
\label{propiso}
Let $\Omega \subset \R^{n}$ be bounded, open set such that $\Omega\not\in \mc B(|\Omega|)$. Then there exists $A_{\alpha}=B_{1}\cup B_{2}\in\mc B(|\Omega|)$ such that
\[
\mu(\Omega,\alpha) > \mu(A_{\alpha},\alpha)=\min_{A\in \mc B(|\Omega|)} \mu(A,\alpha).
\]
Moreover, 
\begin{equation}
\label{minmin}
\mu(A_{\alpha},\alpha)= \mc Q(v_{1}\chi_{B_{1}}-v_{2}\chi_{B_{2}},\alpha),
\end{equation}
for some nonnegative function $v_{1}$ and $v_{2}$, radially decreasing in $B_{1}$ and $B_{2}$ respectively.  More precisely, either $v_{2} \equiv 0$, and $v_{1}$ is positive in $B_{1}=\Omega$, or $v_{1}>0$ in $B_{1}$ and $v_{2}>0$ in $B_{2}$.
\end{prop}
\begin{proof}
Let $u\in W_{0}^{1,2}(\Omega)$ be a minimizer of \eqref{pb}. Using \eqref{lp} and P\'olya-Szeg\"o principle \eqref{ps}, we have that
\begin{align}\notag
\mu(\Omega,\alpha)&= 
\frac{
\ds \int_{\Omega_{+}}|\nabla u_{+}|^{2}\dx +\ds \int_{\Omega_{-}}|\nabla u_{-}|^{2}\dx+
\alpha\left| \int_{\Omega_{+}}u_{+}^{2}\dx-\int_{\Omega_{-}}u_{-}^{2}\dx   \right|}
{\ds
\int_{\Omega_{+}}u_{+}^{2}\dx+\int_{\Omega_{-}}u_{-}^{2}\dx
}\\[.1cm] \label{chain1}
&\ge
\frac{
\ds \int_{\Omega_{+}\diesis}|\nabla u\diesis_{+}|^{2}\dx +\ds \int_{\Omega_{-}\diesis}|\nabla u_{-}\diesis|^{2}\dx+
\alpha\left| \int_{\Omega_{+}\diesis}(u_{+}\diesis)^{2}\dx-\int_{\Omega_{-}\diesis}(u_{-}\diesis)^{2}\dx   \right|
}
{\ds \int_{\Omega_{+}\diesis}(u_{+}\diesis)^{2}\dx+\int_{\Omega_{-}\diesis}(u_{-}\diesis)^{2}\dx 
}
\\[.1cm] \notag
& \ge
\min_{
\substack{w\in W_{0}^{1,2}(\Omega_{+}\diesis)\\z\in W_{0}^{1,2}(\Omega_{-}\diesis)}
}\frac{
\ds \int_{\Omega_{+}\diesis}|\nabla w|^{2}\dx +\ds \int_{\Omega_{-}\diesis}|\nabla z|^{2}\dx+
\alpha\left| \int_{\Omega_{+}\diesis}|w|w\dx+\int_{\Omega_{-}\diesis}|z|z\dx   \right|
}
{\ds
\int_{\Omega_{+}\diesis} |w|^2\dx+\int_{\Omega_{-}\diesis}|z|^2\dx  
}
\\[.1cm] \label{chain3}
&
\ge 
\inf_{A\in\mc B(|\Omega|)} \mu(A;\alpha).
\end{align}
If $u_{+}$ or $u_{-}$ is not radially symmetric, then the inequality \eqref{chain1} is strict. Moreover, if $u_{+}$ and $u_{-}$ are both radially decreasing functions, then $\Omega_{+}$ and $\Omega_{-}$ are balls such that $|\Omega_{+}|+|\Omega_{-}|<|\Omega|$, being $\Omega\not \in \mc B(|\Omega|)$. The monotonicity of $\mu(\,\cdot\,;\alpha)$ with respect to homotheties gives that in this case \eqref{chain3} is strict.

The arguments just used also give \eqref{minmin}.
\end{proof}
In order to conclude the proof of Theorem \ref{mainthm}, we recall an isoperimetric inequality for $\lambda_{T}(\Omega)$ given in \cite{n12}, which assures that if $B_{1}$, $B_{2}$ are disjoint balls with $|B_{1}|=|B_{2}|=|\Omega|/2$, then
\[
\lambda_{T}(\Omega) \ge \lambda_{T}(B_{1}\cup B_{2}).
\]
\begin{proof}[Proof of Theorem \ref{mainthm}]
If $\alpha\le 0$, being $\mu(\Omega,\alpha)=\lambda_\Delta (\Omega)+\alpha$
the result is given by the well-known Faber-Krahn inequality, which follows immediately from the P\'olya-Szeg\"o principle and the properties of rearrangements:
\begin{equation}
\label{naza}
\mc Q(u,\alpha) \ge \mc Q(u\diesis,\alpha)\ge \mu(\Omega\diesis,\alpha). 
\end{equation}
Then we can assume that $\alpha> 0$. 

Proposition \ref{propiso} allows to restrict to the case $\Omega\in \mc B(|\Omega|)$. We denote by $\Omega_{d}$ the union of two disjoint balls with same measure, equal to $|\Omega|/2$.

Then the proof is completed by observing that, by Proposition \ref{charballs}, and the Faber-Krahn inequality and \eqref{naza}, each eigencurve $\alpha\mapsto\mu(\Omega,\alpha)$, $\alpha\ge 0$, is such that $\mu(\Omega,0)\ge \mu(\Omega\diesis,0)=\lambda_\Delta (\Omega\diesis)$, then it increases linearly until it reaches the value $\lambda_{T}(\Omega)$ which is greater than $\lambda_{T}(\Omega_{d})$ (see also Figure \ref{1}). More precisely, the eigencurve $\alpha\mapsto\mu(\Omega,\alpha)$ is above the curve
\[
\alpha\mapsto 
	\begin{cases} 
	\mu(\Omega\diesis,\alpha) &\text{if }\alpha |\Omega|^{2/n} \le \alpha_{c},\\
	\mu(\Omega_{d},\alpha) &\text{if }\alpha |\Omega|^{2/n} \ge \alpha_{c},
	\end{cases}
\]
obtaining \eqref{mainin}.
\end{proof}

\begin{figure}[h]
\begin{tikzpicture}[scale=.5,line cap=round,line join=round,>=triangle 45,x=1.2cm,y=1.2cm]
\draw[-stealth,color=black] (-.5,0) -- (16,0) node[anchor=north west] {$\alpha$};
\draw[-stealth,color=black] (0,-.5) -- (0,13);
\draw (-.8,1.7) circle (.7cm); 
\draw (-2.1,1.7) node[anchor=east]{$\mu(\Omega\diesis,0)=\lambda_\Delta (\Omega\diesis)$};
\draw (-2.2,7) node[anchor=east]{$\mu(\Omega_{d},0)=\lambda_{T}\left(\Omega_{d}\right)$};
\draw (-1.8,6) circle (.41cm); \draw(-.8,6) circle (.50cm);
\draw[loosely dashed,thick] (0,6.4) -- (1.5,7.9);
\draw[loosely dashed,thick] (1.5,7.9) -- (16,7.9);
\draw (-.8,7) circle (.49cm); \draw (-1.8,7) circle (.49cm);
\draw (-.1,2) -- (0.05,2);
\draw[dashed,thick] (-.1,7) -- (5,7);
\draw[dotted] (5,-.1) node[anchor=north]{$\alpha_{c}$} -- (5,7);
\draw[very thick] (0,2) -- (5,7);
\draw[dashed] (4,6) -- (7,9);
\draw[dotted] (7,9) -- (7,-.1) node[anchor=north]{$\tilde\alpha_{\Omega\diesis}$};
\draw[dashed] (7,9) -- (16,9);
\draw[very thick] (5,7) -- (16,7);
  \path (9,9.05) coordinate (n);
      \draw[black] (11,10)  node (a) {
        $\lambda_{T}(\Omega\diesis)$};
      \path[black,-stealth] (a) edge[in=50,out=170] (n);
\end{tikzpicture}
\caption{A scheme of the eigencurves $\alpha\mapsto\lambda(\Omega,\alpha)$, $\alpha\ge0$, when $|\Omega|=\kappa$ is fixed. If $\Omega$ corresponds to $\Omega_{d}$ union of two disjoint balls of equal measure, then $\mu(\Omega_{d},\alpha)$ is constant for $\alpha \ge 0$. Otherwise, $\mu(\Omega,\alpha)$ increases until it reaches its maximum value $\mu(\Omega,\alpha_{\Omega})=\lambda_{T}(\Omega)$ in $\tilde\alpha_{\Omega}=\lambda_{T}(\Omega)-\lambda_{\Delta}(\Omega)$, then it is constant for $\alpha\ge \tilde\alpha_{\Omega}$. The solid line represents the values of $\ds\min_{|\Omega|=\kappa}\mu(\Omega,\alpha)$.}
\label{1}
\end{figure}
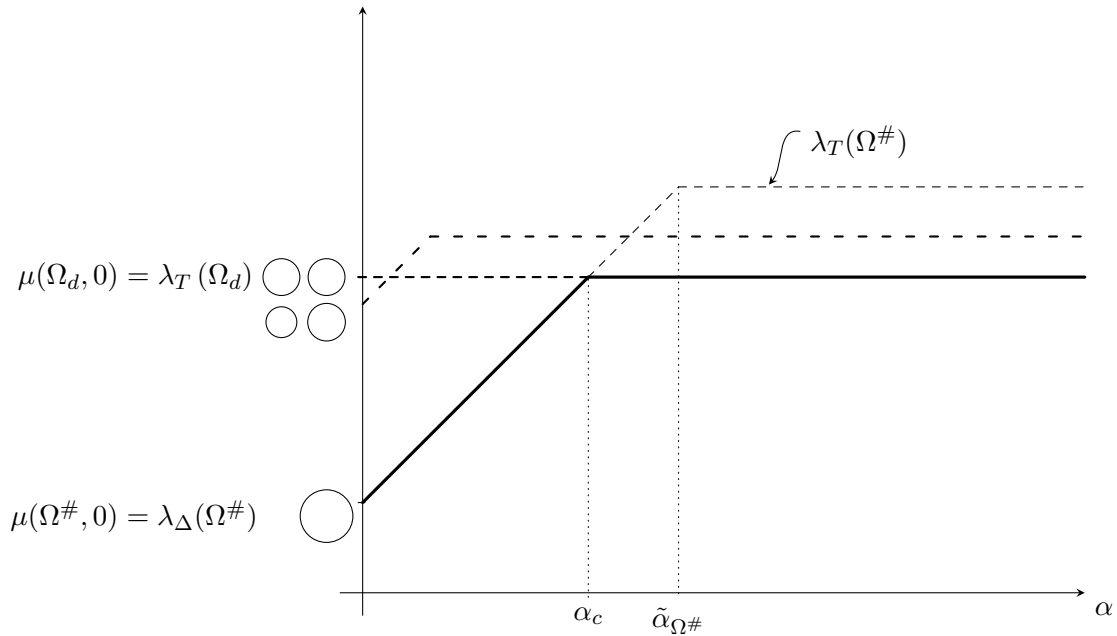

\thanks{{\bf Acknowledgement.} The author has been partially supported by 
Gruppo Nazionale per l'Analisi Matematica, la Probabilit\`a e le loro Applicazioni (GNAMPA) of the Istituto Nazionale di Alta Matematica (INdAM), project ``Dise\-guaglianze funzionali e problemi so\-vra\-de\-ter\-mi\-na\-ti'' 2013.}


\begin{thebibliography}{20}
\bibitem{bdnt14}
  B.~Brandolini, F.~Della Pietra, C.~Nitsch and C.~Trombetti,
  \newblock {Symmetry breaking in a constrained Cheeger type
  isoperimetric inequality},
  \newblock {\em ESAIM Control Optim. Calc. Var.}, to appear.
\bibitem{bfnt11}
  B.~Brandolini, P.~Freitas, C.~Nitsch and C.~Trombetti,
  \newblock {Sharp estimates and saturation phenomena for a nonlocal
    eigenvalue problem},
  \newblock {\em Adv. Math.}, {\bfseries 228} (2011), 2352--2365.
  \bibitem{brkj}
  L.~Brasco.
  \newblock {On torsional rigidity and principal frequencies:
an invitation to the Kohler-Jobin rearrangement
technique},
  \newblock {\em ESAIM Control Optim. Calc. Var.},
  \newblock {in press, DOI:10.1051/cocv/2013065.}
  \bibitem{bz88}
J.~E. Brothers and W.~P.~. Ziemer,
\newblock {Minimal rearrangements of {S}obolev functions},
\newblock {\em J. Reine Angew. Math.}, {\bfseries 384} (1988), 153--179.
\bibitem{chdb12}
F.~Chiacchio and G.~{di Blasio},
\newblock {Isoperimetric inequalities for the first {N}eumann eigenvalue in
  {G}auss space},
\newblock {\em Annales de l'Institut Henri Poincar\'e (C) Non Linear Analysis},
  {\bfseries 29} (2012), 199--216.
\bibitem{dpg2}
F.~{Della Pietra} and N.~Gavitone,
\newblock {Symmetrization for Neumann anisotropic problems and related
  questions},
\newblock {\em Advanced Nonlinear Stud.}, {\bfseries 12} (2012), 219--235.
\bibitem{dpg1}
F.~{Della Pietra} and N.~Gavitone.
\newblock {Relative isoperimetric inequality in the plane: the
  anisotropic case}.
\newblock {\em J. Convex. Anal.}, {\bfseries 20} (2013), 157--180. 
  \bibitem{dpgtors}
F.~{Della Pietra} and N.~Gavitone,
\newblock {Sharp bounds for the first eigenvalue and the torsional rigidity
  related to some anisotropic operators},
\newblock {\em Math. Nachr.}, {\bfseries 287} (2014), 194--209.
\bibitem{dpgccm}
F.~{Della Pietra} and N.~Gavitone,
\newblock {Stability results for some fully nonlinear eigenvalue estimates},
\newblock {\em Comm. Contemporary Math.},
\newblock {in press, DOI:10.1142/S0219199713500399}.
\bibitem{dpgmonge}
F.~{Della Pietra} and N.~Gavitone,
\newblock {Upper bounds for the eigenvalues of Hessian equations},
\newblock {\em Annali Mat. Pura Appl.},
\newblock {in press, DOI:10.1007/s10231-012-0307-5}.
\bibitem{dpgrobin}
F.~{Della Pietra} and N.~Gavitone,
\newblock {Faber-Krahn inequality for anisotropic eigenvalue problems with Robin boundary conditions},
\newblock{arXiv:1311.3456}.
\bibitem{efknt} L. Esposito, V. Ferone, B. Kawohl, C. Nitsch and
  C. Trombetti,
  \newblock The longest shortest fence and sharp Poincar\'e-Sobolev
  inequalities, 
  \newblock{\em Arch. Rational Mech. Anal.}, {\bfseries 206} (2012), 821--851.
\bibitem{gn11}
  I.V.~Gerasimov and A.I.~Nazarov.
  \newblock {Best constant in a three-parameter Poincar\'e
    inequality}.\
  \newblock {\em Probl. Mat. Anal.}, {\bfseries 61} (2011), 69--86 (Russian). English transl.: {\em J. Math. Sci.}, 
  {\bfseries 179} (2011), 80--99. 
\bibitem{hen} A. Henrot, Extremum problems for eigenvalues of elliptic operators,
Frontiers in Mathematics, {\it Birkhauser}, Basel, 2006. 
\bibitem{k} B. Kawohl, Rearrangements and convexity of level sets
    in P.D.E., Lecture notes in mathematics {\bfseries 1150}, {\it Springer
  Verlag}, Berlin, New York, 1985.  
\bibitem{ma}
   V.G. Maz'ya.
   \newblock {Sobolev spaces with applications to elliptic partial
     differential equations}. 
   \newblock {\em Springer Verlag}, Heidelberg, 2011.
\bibitem{n12}
   A.I.~Nazarov.
   \newblock {On symmetry and asymmetry in a problem of shape
     optimization}. 
   \newblock http://arxiv.org/abs/1208.3640 (2012), 1--5.
\end{thebibliography}
\end{document}